\documentclass[a4paper,12pt]{article}
\usepackage{amssymb,amsmath,amsthm,times,mathrsfs,fullpage}
\usepackage[pdftex]{hyperref}
\hypersetup{plainpages=True, pdfstartview=FitH, bookmarksopen=true,
colorlinks=true,linkcolor=blue,citecolor=blue}
\usepackage{tikz,pgf}
\usepackage{epstopdf}
\usepackage{array}
\usepackage{cite}
\usepackage{lscape}

\usepackage{todonotes}

\newcommand{\LandauO}{\mathcal{O}}
\newcommand{\Ai}{\text{\normalfont Ai}}

\def\qed{{\quad\rule{1mm}{3mm}\,}}

\begin{document}

\newtheorem{thm}{Theorem}[section]
\newtheorem{cor}[thm]{Corollary}
\newtheorem{lmm}[thm]{Lemma}
\newtheorem{conj}[thm]{Conjecture}
\newtheorem{pro}[thm]{Proposition}
\newtheorem{df}[thm]{Definition}
\theoremstyle{remark}
\newtheorem{rem}[thm]{Remark}

\newcommand*\samethanks[1][\value{footnote}]{\footnotemark[#1]}
\newcommand*{\myand}{\,\and\,}

\title{\textbf{Enumeration of $d$-combining \\ Tree-Child Networks}}
\author{Yu-Sheng Chang\thanks{Department of Mathematical Sciences, National Chengchi University, Taipei 116, Taiwan.}
\myand
Michael Fuchs\samethanks
\myand
Hexuan Liu\samethanks
\myand
Michael Wallner\thanks{Institute for Discrete Mathematics and Geometry, TU Wien, 1040 Vienna, Austria.}
\myand
Guan-Ru Yu\thanks{Department of Mathematics, National Kaohsiung Normal University, Kaohsiung 824, Taiwan.}}
\date{\today}
 \maketitle

\begin{abstract}
Tree-child networks are one of the most prominent network classes for modeling evolutionary processes which contain reticulation events. Several recent studies have addressed counting questions for {\it bicombining tree-child networks} which are tree-child networks with every reticulation node having exactly two parents. In this paper, we extend these studies to {\it $d$-combining tree-child networks} where every reticulation node has now $d\geq 2$ parents. Moreover, we also give results and conjectures on the distributional behavior of the number of reticulation nodes of a network which is drawn uniformly at random from the set of all tree-child networks with the same number of leaves.
\end{abstract}

\section{Introduction and Results}

The evolutionary process of, e.g., chromosomes, species, and populations is not always tree-like due to the occurrence of reticulation events caused by meiotic recombination on the chromosome level,
specification and horizontal gene transfer on the species level,
and sexual recombination on the population level. Because of this, {\it phylogenetic networks} have been introduced as appropriate models for reticulate evolution. Studying the properties of these networks is at the moment one of the most active areas of research in phylogenetics; see \cite{HuRuSc,St}.

While algorithmic and combinatorial aspects of phylogenetic networks have been investigated now for a few decades, enumerating and counting phylogenetic networks as well as understanding their ``typical shape'' are relatively recent areas of research; see~\cite[page~253]{St} where such questions are only discussed in one short paragraph. However, the last couple of years have seen a lot of progress on these questions, in particular for the class of  {\it tree-child networks}, which is one of the most prominent subclasses amongst the many subclasses of phylogenetic networks; see \cite{CaZh,DiSeWe,FuGiMa1,FuGiMa2,FuHuYu,FuLiYu,FuYuZh,PoBa}.

Most of the studies on tree-child networks have focused on {\it bicombining tree-child networks} which are tree-child networks where every reticulation event involves exactly two individuals.
The purpose of this extended abstract is to discuss extensions of previous results to {\it multicombining tree-child} networks. More precisely, we will focus on {\it $d$-combining tree-child networks} which are tree-child networks whose reticulation events involve exactly $d\geq 2$ individuals. We will highlight similarities and differences between the two cases $d=2$ and $d>2$.

Before explaining our results, we will first give precise definitions and fix some notation. We start with the definition of phylogenetic networks.

\begin{figure}
\centering
\includegraphics[scale=0.9]{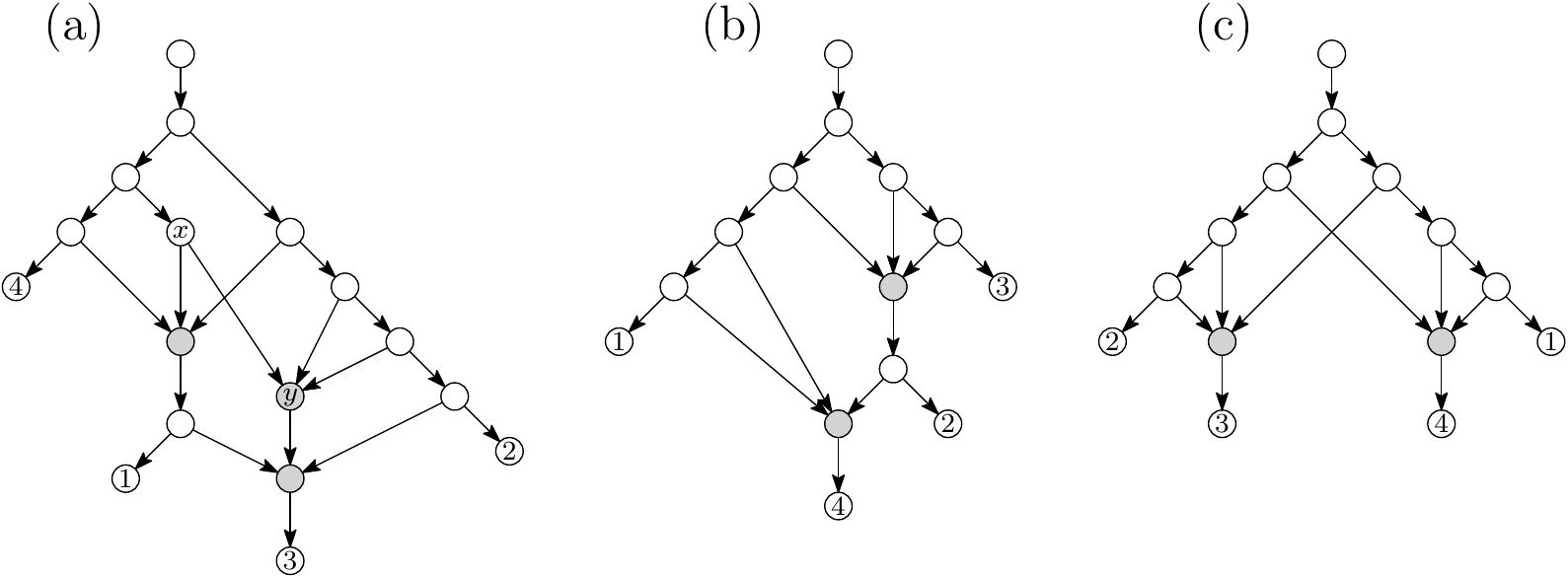}
\caption{(a) A $3$-combining phylogenetic network which is \emph{not} a tree-child network (because both children of the tree node~$x$ are reticulation nodes and the only child of the reticulation node~$y$ is also a reticulation node); (b) a $3$-combining tree-child network; (c) a $3$-combining one-component tree-child network.}\label{pn-fig}
\end{figure}

\begin{df}[Phylogenetic network]
A (rooted) phylogenetic network with $n$ leaves is a rooted, simple, directed acyclic graph (DAG) with no nodes of in- and outdegree~$1$ and exactly $n$ nodes of indegree $1$ and outdegree $0$ (i.e., leaves) which are bijectively labeled with labels from the set $\{1,\ldots,n\}$.
\end{df}

In this work, we will only consider phylogenetic networks whose nodes all have either indegree $1$ {\it or} outdegree $1$ and whose internal nodes (i.e., nodes which are neither leaves nor the root) with indegree $1$ all have outdegree $2$ (bifurcating case); the latter nodes will be called {\it tree nodes}. Finally, we will assume that all internal nodes with outdegree $1$ have indegree $d\geq 2$ and these nodes will be called {\it reticulation nodes}; see Figure~\ref{pn-fig} for examples with $d=3$. Note that $d=2$ is the above mentioned bicombining case.

We next recall the definition of tree-child networks.
\begin{df}[Tree-child network]
A phylogenetic network is called a tree-child network if every non-leaf node has at least one child which is not a reticulation node.
\end{df}

In other words, a phylogenetic network is a tree-child network if (a) the root is not followed by a reticulation node; (b) a reticulation node is not followed by another reticulation node; and (c) a tree node has at least one child which is not a reticulation node; see Figure~\ref{pn-fig}, (b) for an example.
A simple and important subclass of tree-child networks is the class of one-component tree-child networks; see the definition below and  Figure~\ref{pn-fig}, (c) for an example.

\begin{df}[One-component tree-child network]
A tree-child network is called a one-component tree-child network if every reticulation node is directly followed by a leaf.
\end{df}

One-component networks are more ``tree-like'' than general tree-child networks. Moreover, they constitute an important building block in the construction of general tree-child networks; see \cite{CaZh} for the bicombining case.

In the following, we denote by $\mathrm{OTC}^{(d)}_{n,k}$ and $\mathrm{TC}^{(d)}_{n,k}$ the number of one-component and general $d$-combining tree-child networks with $n$ leaves and $k$ reticulation nodes, respectively.
Note that the tree-child property implies that $k\leq n-1$. Thus, the total numbers of one-component and general $d$-combining tree-child networks, denoted by $\mathrm{OTC}^{(d)}_n$ and $\mathrm{TC}^{(d)}_n$, satisfy
\[
\mathrm{OTC}^{(d)}_n:=\sum_{k=0}^{n-1}\mathrm{OTC}^{(d)}_{n,k}\qquad\text{and}\qquad\mathrm{TC}^{(d)}_n:=\sum_{k=0}^{n-1}\mathrm{TC}^{(d)}_{n,k}.
\]

Now, we are ready to present our results. First, for one-component tree-child networks, we have the following formula which extends the one for $d=2$ from Theorem 13 in \cite{CaZh}.

\begin{thm}\label{formula-OTC}
The numbers of one-component $d$-combining tree-child networks with $n$~leaves and $k$~reticulation nodes for $0 \leq k \leq n-1$ are given by
\[
\mathrm{OTC}^{(d)}_{n,k}=\binom{n}{k}\frac{(2n+(d-2)k-2)!}{(d!)^k \, 2^{n-k-1} \,(n-k-1)!}
\]
and equal to $0$ otherwise.
\end{thm}

Using this formula, we obtain the following consequence.
\begin{cor}\label{Cor-1}
Let $R^{(d)}_n$ be the number of reticulation nodes of a one-component $d$-combining tree-child network picked uniformly at random from the set of all one-component $d$-combining tree-child networks with $n$ leaves. Then, we have the following limit behavior of $R^{(d)}_n$.
\begin{itemize}
\item[(i)] For $d=2$ (bicombining case), we have the weak convergence result:
\[
\frac{R^{(2)}_n-n+\sqrt{n}}{\sqrt[4]{n/4}}\stackrel{w}{\longrightarrow}N(0,1).
\]
\item[(ii)] For $d=3$, we have
\[
n-1-R^{(3)}_n\stackrel{w}{\longrightarrow}\operatorname{Bessel}(1,2),
\]
where Bessel$(v,a)$ denotes the Bessel distribution, i.e.,
\[
\mathbb{P}(\operatorname{Bessel}(1,2)=k)=\frac{1}{I_1(2)k!(k+1)!},\qquad (k\geq 0).
\]
Here, $I_v(a)=\left(\frac{a}{2}\right)^v\sum_{k=0}^{\infty}\frac{1}{k!\Gamma(k+v+1)}\frac{a^{2k}}{4^k}$ is the modified Bessel function of the first kind.
\item[(iii)] For $d\geq 4$, the limit law of $n-1-R_n^{(d)}$ is degenerate, i.e.,
\[
n-1-R^{(d)}_n\stackrel{w}{\longrightarrow} 0.
\]
\end{itemize}
\end{cor}
\begin{rem}
If $t$ denotes the number of tree nodes and $N$ the total number of nodes, then by the handshaking lemma, we have
\begin{equation}\label{tree-nodes}
t=n+(d-1)k-1\qquad\text{and}\qquad N=2n+dk.
\end{equation}
Therefore, we have similar limit distribution results for these numbers as well.
\end{rem}

Note that the above result for $d=2$ is already contained in the proof of Theorem 3 in \cite{FuYuZh} where even a local limit theorem was proved; see also \cite{FuLiYu}.
Using the above corollary, we now obtain the first order asymptotics of the total number of one-component tree-child networks.

\begin{cor}\label{Cor-2}
\begin{itemize}
\item[(i)] For $d=2$ (bicombining case), we have
\[
\mathrm{OTC}^{(2)}_n\sim \frac{1}{4 \pi \sqrt{ e}} (n!)^2 2^{n} e^{2\sqrt{n}} n^{-9/4}.
\]


\item[(ii)] For $d=3$, we have
\[
\mathrm{OTC}^{(3)}_n\sim I_1(2)\cdot\mathrm{OTC}^{(3)}_{n,n-1}\sim\frac{I_1(2)\sqrt{3}}{9\pi}(n!)^3\left(\frac{9}{2}\right)^nn^{-3},
\]
where $I_{v}(\alpha)$ is as in Corollary~\ref{Cor-1}, (ii).
\item[(iii)] For $d\geq 4$, we have
\[
\mathrm{OTC}^{(d)}_n\sim\mathrm{OTC}^{(d)}_{n,n-1}\sim\frac{d!}{d^{d-1/2}(2\pi)^{(d-1)/2}}(n!)^{d}\left(\frac{d^d}{d!}\right)^nn^{3(1-d)/2}.
\]
\end{itemize}
\end{cor}


Again, the result for the case $d=2$ is already contained in \cite{FuYuZh}.
It is also the only case of the three above in which we find a stretched exponential in the asymptotics; see~\cite{ElFaWa}.

We next turn to general tree-child networks. Here, in contrast to one-component tree-child networks, we only understand the behavior of $\mathrm{TC}^{(d)}_{n,k}$ for fixed $k$ and for $k=n-1$ (i.e., for maximally reticulated networks).

First, for fixed $k$, the first order asymptotics in the bicombining case ($d=2$) was derived in \cite{FuGiMa1,FuGiMa2} and with a different method in \cite{FuHuYu}; see also \cite{PoBa} where the result was re-derived with yet another method which is however based on a (yet) unproven conjecture. The approach from \cite{FuHuYu} can also be used in the $d$-combining case leading to the following result.

\begin{thm}\label{gen-tc-fixed-k}
For the number of $d$-combining tree-child networks with $n$ leaves and $k$ reticulation nodes, we have for fixed $k$, as $n\rightarrow\infty$,
\[
\mathrm{TC}^{(d)}_{n,k}\sim\frac{2^{(4-d)k-1}}{d!^{k}\,k!\sqrt{\pi}}n!2^nn^{(4-d)k-3/2}.
\]
\end{thm}
\begin{rem}
Our approach can also be used to compute tables of $\mathrm{TC}^{(d)}_{n,k}$ for small values of $n$, $k$, and $d$; see the Appendix. Moreover, the approach is also capable of giving exact formulas for $\mathrm{TC}^{(d)}_{n,k}$ and small values of $k$; see \cite{CaZh,FuGiMa2} for formulas in the bicombining case.
\end{rem}

Note that the asymptotic order in the above theorem is much smaller than the one obtained for one-component tree-child networks. Thus, the majority of tree-child networks do not have a bounded number of reticulation nodes. In fact, the number of reticulation nodes of a ``typical'' tree-child network is close to the maximum $n-1$. More precisely, the following result holds.

\begin{thm}\label{gen-tc-max-k}
For the number of $d$-combining tree-child networks with $n$ leaves, we have
\[
\mathrm{TC}^{(d)}_{n}
=\Theta\left(\mathrm{TC}^{(d)}_{n,n-1}\right)
=\Theta\left((n!)^{d} \, \gamma(d)^n \, e^{3a_1\beta(d)n^{1/3}} n^{\alpha(d)} \right),
\]
where $a_1=-2.33810741\ldots$ is the largest root of the Airy function of the first kind and
\[
\alpha(d)=-\frac{d(3d-1)}{2(d+1)},
\qquad\beta(d)=\left(\frac{d-1}{d+1}\right)^{2/3},
\qquad\gamma(d)=4\frac{(d+1)^{d-1}}{(d-1)!}.
\]
\end{thm}

\pagebreak

For the bicombining case, this result was proved in \cite{FuYuZh} by encoding tree-child networks with $n$ leaves and $n-1$ reticulation nodes by certain words and (asymptotically) counting these words with the method from \cite{ElFaWa}. In the more general $d$-combining case, we will use a similar strategy, however, details will be more demanding
due to the dependence on the parameter $d$.

As for the stochastic behavior of the number of reticulation nodes for general tree-child networks, we have a conjecture for the limit laws which we are going to present in Section~\ref{con+rem}. Note that even for $d=2$, no limit law for any shape parameter of random tree-child networks has been established yet.

Our conjecture will clarify the behavior of $\mathrm{TC}^{(d)}_{n,k}$ for $k$ close to $n$.
Thus, the behavior of $\mathrm{TC}^{(d)}_{n,k}$ for small and large $k$ is clear.
For the remaining range, there is an interesting recent conjecture for the bicombining case in~\cite{PoBa}, which has been proved for the special case of one-component tree-child networks in~\cite{FuLiYu}.
Whether this conjecture can be extended to $d$-combining tree-child networks is not clear yet; see the comments in Section~\ref{con+rem}.

We conclude the introduction with an outline of this extended abstract. In the next section, we will consider one-component networks and prove Theorem~\ref{formula-OTC} and Corollaries~\ref{Cor-1} and~\ref{Cor-2}. In Section~\ref{gen-net}, we will discuss our results for general networks. Finally, Section~\ref{con+rem} will contain the above mentioned conjecture for the limit laws of the number of reticulation nodes of a random $d$-combining tree-child network and some concluding remarks.

\section{One-Component Networks}

In this section, we will prove our results for one-component tree-child networks.
We start with Theorem~\ref{formula-OTC}.

\vspace*{0.3cm}\noindent{\it Proof of Theorem~\ref{formula-OTC}.} Suppose $N$ is a one-component $d$-combining tree-child network with $n-1$ leaves and $k-1$ reticulation nodes.

Then, we can construct one-component $d$-combining tree-child networks with $n$ leaves and $k$ reticulation nodes from $N$ by the following three steps:
(i) put $d$ new nodes into the {\it candidate edges} where we call an edge of $N$ a candidate edge if it is not incident to any reticulation node;
(ii) create a new reticulation node which is adjacent to the $d$ new nodes; and
(iii) add a new leaf as a child of this reticulation node;
moreover, label it with a label from $\{1,\ldots,n\}$ and increase all (old) labels in $N$ which are at least as large as the new label by $+1$ (if there are any).

Now, note that in step (i), we have
\[
\underbrace{n-1+(d-1)(k-1)-1}_{\substack{\text{\# edges leading to} \\ \text{a tree node; see (\ref{tree-nodes})}}}+\underbrace{n-1}_{\substack{\text{\# edges leading} \\ \text{to a leaf} }}-\underbrace{(k-1)}_{\substack{\text{\# edges below} \\ \text{ret.\ nodes}}}=2n+(d-2)(k-1)-3
\]
candidate edges and thus there are
\[
\binom{2n+(d-2)k-2}{d}
\]
choices of the $d$ nodes. Moreover, in step (iii), there are $n$ choices of the label. Finally, note that the above construction gives each network exactly $k$ times.

Overall, the above arguments give
\[
\mathrm{OTC}^{(d)}_{n,k}=\frac{n}{k}\binom{2n+(d-2)k-2}{d}\mathrm{OTC}^{(d)}_{n-1,k-1},
\]
and by iteration,
\[
\mathrm{OTC}^{(d)}_{n,k}=\binom{n}{k}\frac{(2n+(d-2)k-2)!}{d!^k(2n-k-2)!}\mathrm{OTC}^{(d)}_{n-k,0}.
\]
The result follows now by the fact that
\[
\mathrm{OTC}^{(d)}_{n-k,0}=(2(n-k)-3)!!=\frac{(2n-2k-2)!}{2^{n-k-1}(n-k-1)!}
\]
since $\mathrm{OTC}^{(d)}_{n-k,0}$ is the number of phylogenetic trees with $n-k$ leaves; e.g., see \cite[Section 2.1]{St}. \qed

Now, we can prove the two corollaries from above.

\vspace*{0.35cm}\noindent{\it Proof of Corollaries~\ref{Cor-1} and~\ref{Cor-2}.} Since the results for $d=2$ are already contained in \cite{FuYuZh} (see also \cite{FuLiYu}), we can focus on the cases $d\geq 3$.

We start with the case $d=3$. Note that
\[
\mathrm{OTC}^{(3)}_{n,k}=\binom{n}{k}\frac{(2n+k-2)!}{3^k2^{n-1}(n-k-1)!},\qquad (0\leq k\leq n-1)
\]
and this sequence is increasing in $k$. (This is in contrast to $d=2$ where this sequence has a maximum at $k=n-\sqrt{n+1}$; see \cite{FuYuZh}.) By replacing $k$ by $n-1-k$ and using Stirling's formula, we obtain that
\begin{equation}\label{exp-otc}
\mathrm{OTC}^{(3)}_{n,n-1-k}=\frac{1}{k!(k+1)!}\cdot\frac{n(3n-3)!}{6^{n-1}}\left(1+{\mathcal O}\left(\frac{1+k^2}{n}\right)\right)
\end{equation}
uniformly for $k$ with $k=o(\sqrt{n})$. Thus, by a standard application of the Laplace method:
\[
\mathrm{OTC}_n^{(3)}\sim\left(\sum_{k\geq 0}\frac{1}{k!(k+1)!}\right)\cdot\frac{n(3n-3)!}{6^{n-1}}=I_1(2)\cdot\frac{n(3n-3)!}{6^{n-1}}
\]
which is the first claim from Corollary~\ref{Cor-2}, (ii); the second follows from this by another application of Stirling's formula. Moreover, since
\[
\mathbb{P}(R_n^{(3)}=n-1-k)=\frac{\mathrm{OTC}^{(3)}_{n,n-1-k}}{\mathrm{OTC}_n^{(3)}},
\]
the result from Corollary~\ref{Cor-1}, (ii) follows from the above two expansions too.

Next, we consider the case $d\geq 4$.
The details of the proof are the same as above, with the main difference that the expansion~\eqref{exp-otc} now becomes
\[
\mathrm{OTC}^{(d)}_{n,n-1-k}=\left(\frac{d^2d!}{2d^d}\right)^k\frac{1}{k!(k+1)!}\cdot n^{(3-d)k}\cdot\frac{n(dn-d)!}{d!^{n-1}}\left(1+{\mathcal O}\left(\frac{1+k^2}{n}\right)\right)
\]
uniformly for $k$ with $k=o(\sqrt{n})$. This expansion, for $d\geq 4$, contains the (non-trivial decreasing) factor $n^{(3-d)k}$ which is responsible for $\mathrm{OTC}_n^{(d)}$ being now asymptotically dominated by $\mathrm{OTC}^{(d)}_{n,n-1}$ (proving Corollary~\ref{Cor-2}, (iii)) and the limiting distribution of $n-1-R_n^{(d)}$ being degenerate (proving Corollary~\ref{Cor-1}, (iii)).
\qed


\section{General Networks}\label{gen-net}

In this section, we will discuss the asymptotic enumeration of general $d$-combining tree-child networks with a fixed 
(Theorem~\ref{gen-tc-fixed-k}) and a maximal number of reticulation nodes (Theorem~\ref{gen-tc-max-k}). Note that the latter dominates asymptotically all networks of a given size.

We start with Theorem~\ref{gen-tc-fixed-k} on $d$-combining networks with a fixed number of reticulation nodes.
It can be proved by generalizing the approach from~\cite{FuHuYu}, which was based on the classification of tree-child networks via {\it component graphs} from \cite{CaZh}.
Component graphs can also be defined for $d$-combining tree-child networks and then be used to prove Theorem~\ref{gen-tc-fixed-k}; details will be given in the journal version of this work.
Moreover, component graphs can also be used as in \cite{CaZh} to
(a) compute $\mathrm{TC}^{(d)}_{n,k}$ for small values of $n$, $k$, and $d$ (see the Appendix); and
(b) give explicit formulas for small values of $k$.

The remainder of this section is devoted to the proof of Theorem~\ref{gen-tc-max-k}, which extends the approaches from \cite{ElFaWa,FuYuZh}.
We start with some lemmas which are generalizations of the corresponding results from \cite{FuYuZh} (and are proved with similar arguments).
The first lemma shows that the asymptotic growth of $\mathrm{TC}^{(d)}_n$ is, up to a constant, determined by the asymptotics of $\mathrm{TC}^{(d)}_{n,n-1}$.

\begin{lmm}
\label{lem:TC-Theta-TCnn1}
For $n \to \infty$, we have
\[
\mathrm{TC}^{(d)}_n=\Theta\left(\mathrm{TC}^{(d)}_{n,n-1}\right).
\]
\end{lmm}

\begin{proof}
Let $N$ be a $d$-combining tree-child network with $n$ leaves and $k$ reticulation nodes.
A {\it free tree node} is a tree node whose children are both not reticulation nodes; the edges to these children are called {\it free edges}. Using a simple counting argument, it is easy to see that $N$ has $2(n-k-1)$ free edges; see \cite[Lemma~1]{FuYuZh} for the case $d=2$.

Next, we can construct $d$-combining tree-child networks with $n$ leaves and $k+1$ reticulation nodes by (i) inserting $d$ tree nodes into the root edge of $N$ and a reticulation node into a free edge and (ii) connecting the $d$ new tree nodes to the new reticulation node. Note that each network built in this way is different. Thus,
\[
2(n-k-1)\mathrm{TC}^{(d)}_{n,k}\leq\mathrm{TC}^{(d)}_{n,k+1}.
\]
Iterating this construction yields
\begin{equation}\label{tc-number-ub}
\mathrm{TC}^{(d)}_{n,k}\leq\frac{1}{2^{n-k-1}(n-k-1)!}\mathrm{TC}^{(d)}_{n,n-1}
\end{equation}
and thus,
\[
\mathrm{TC}^{(d)}_{n,n-1}\leq\mathrm{TC}^{(d)}_{n}\leq\bigg(\sum_{j\geq 0}\frac{1}{2^j j!}\bigg)\cdot\mathrm{TC}^{(d)}_{n,n-1}=
\sqrt{e}\cdot\mathrm{TC}^{(d)}_{n,n-1},
\]
which proves the claim.
\end{proof}

Next, we define a generalization of the class of words from \cite{FuYuZh} which is used to encode $d$-combining tree-child networks with a maximal number of reticulation nodes.

\begin{df}\label{words-tc}
Let $\mathcal{C}_n^{(d)}$ denote the class of words built from $n$ letters $\{\omega_1,\ldots,\omega_n\}$ in which each letter occurs exactly $d+1$ times such that in every prefix the letter $\omega_i$ has either not yet occurred more than $d-2$ times, or, if it has, then the number of occurrences of $\omega_i$ is at least as large as the number of occurrences of $\omega_j$ for all $j>i$.
\end{df}

In \cite{FuYuZh}, a bijection between bicombining tree-child networks with $n$ leaves and $k$ reticulationd nodes whose labels are removed and $\mathcal{C}_{n-1}^{(2)}$ was proved. In fact, this bijection can be extended to $d$-combining networks. Then, because the networks are all different when labeling the (now empty) leaves by a random permutation, we obtain the following lemma.

\begin{lmm}
\label{lem:TC-cn-bij}
Let $c_n^{(d)}$ be the cardinality of $\mathcal{C}_{n}^{(d)}$. Then
\[
\mathrm{TC}_{n,n-1}^{(d)}=n!c_{n-1}^{(d)}.
\]
\end{lmm}

Now the recursive nature of this encoding allows us to derive the following counting result.

\begin{lmm}
\label{lem:cn-bnm}
We have
\begin{align}
\label{eq:cn-bnm}
c_{n}^{(d)}=\sum_{m\geq 1}b_{n,m}^{(d)},
\end{align}
where $b_{n,m}^{(d)}$ satisfies the recurrence
\begin{align}
\label{eq:recbnm}
b_{n,m}^{(d)}=\frac{dn+m-2}{dn+m-d-1}b_{n,m-1}^{(d)}+\binom{dn+m-2}{d-1}b_{n-1,m}^{(d)},\qquad (n\geq 2,0\leq m\leq n)
\end{align}
with
$b_{1,1}^{(d)}=1$ and $b_{n,m}^{(d)}=0$ for (i) $n\geq 2$ and $m=-1$; (ii) $n=1$ and $m=0$; and (iii) $n<m$.
\end{lmm}

\begin{proof}
First, note that any word in $\mathcal{C}_n^{(d)}$ has a suffix $\omega_n\omega_m\omega_{m+1}\cdots\omega_{n-1}\omega_n$ with $1\leq m\leq n$.
Denote by $b_{n,m}^{(d)}$ the number of these words. Removing the $d$ occurrences of $\omega_n$ from these words gives a word of $\mathcal{C}_{n-1}^{(d)}$ with suffix $\omega_m\omega_{m+1}\cdots\omega_{n-1}$, i.e., it has a suffix $\omega_{n-1} \omega_j\omega_{j+1}\dots \omega_{n-1}$ for $j=1,\dots,m$.
Reversing this procedure gives
\begin{align}
\label{eq:bnmsum}
b_{n,m}^{(d)}=\binom{dn+m-2}{d-1}\sum_{j=1}^{m}b_{n-1,j}^{(d)},
\end{align}
where the binomial coefficient counts the number of ways of adding back the $d-1$ occurrences of $\omega_n$ after two $\omega_n$'s have been added, one before the last $\omega_m$ and one at the end of the word.
By Definition~\ref{words-tc} these first $d-1$ occurrences of $\omega_n$ may be anywhere.
Differencing yields
\[
\frac{b_{n,m}^{(d)}}{\binom{dn+m-2}{d-1}}-\frac{b_{n,m-1}^{(d)}}{\binom{dn+m-3}{d-1}}=b_{n-1,m}^{(d)}.
\]
This gives the  claimed recurrence and the initial conditions are easily checked.
\end{proof}

The advantage of the recurrence for $b_{n,m}^{(d)}$ is that we are actually only interested in the asymptotics of $b_{n,n}^{(d)}$ as  by~\eqref{eq:cn-bnm} and \eqref{eq:bnmsum} we have
\begin{align*}
   b_{n,n}^{(d)}
    = \binom{(d+1)n-2}{d-1}c_{n-1}^{(d)}.
\end{align*}

\newcommand{\dd}{e^{(d)}}
\newcommand{\ddh}{\hat{e}^{(d)}}

Now we are ready to use the method from~\cite{ElFaWa}.
Due to the similarities, we will only discuss the main differences.
We start with the following transformation of $(b_{n,m})_{0 \leq m \leq n}$ to $(\dd_{i,j})_{\substack{0 \leq i \leq j \\ i-j \text{ even}}}$, which changes the indices and captures the exponential and superexponential terms coming from the binomial coefficient in~\eqref{eq:recbnm}.

\begin{lmm}\label{mod-rec}
We have
\begin{align*}
     b^{(d)}_{n,m} &= \lambda(d)^n (n!)^{d-1} \dd_{n+m,n-m}
     && \text{ with } &
     \lambda(d) &= \frac{(d+1)^{d-1}}{(d-1)!},
\end{align*}
where $\dd_{n,m}$ satisfies the following recurrence
{%
\small
\begin{align}
    \label{eq:recdnm}
    \dd_{n,m} &= \mu_{n,m}^{(d)} \dd_{n-1,m+1} + \nu_{n,m}^{(d)} \dd_{n-1,m-1}
\end{align}
}%
with
\[
\mu^{(d)}_{n,m}=1+\frac{2(d-1)}{(d+1)n + (d-1)m - 2(d+1)}\quad\text{and}\quad \nu^{(d)}_{n,m}=\prod_{i=2}^d \left(1 - \frac{2(m+i)}{(d+1)(n+m)}\right)
\]
for $n \geq 3$ and $m \geq 0$, where $\dd_{n,-1}=\dd_{2,n}=0$ except for $\dd_{2,0}=1/\lambda(d)$.
\end{lmm}

Note that we are interested in $\dd_{2n,0} = \frac{b^{(d)}_{n,n}}{\lambda(d)^n (n!)^{d-1}}$ because by the previous lemmas
we have
{
\begin{align}
    \label{eq:thetachain}
    \mathrm{TC}^{(d)}_n
    =\Theta\left( \mathrm{TC}^{(d)}_{n,n-1} \right)
    =\Theta\left( n!c_{n-1}^{(d)} \right)
    =\Theta\left( n! \, n^{1-d} \, b_{n,n}^{(d)} \right)
    =\Theta\left( (n!)^d \, \lambda(d)^n \, n^{1-d} \, \dd_{2n,0} \right)
    .
\end{align}
}%
Moreover, observe that for the Theta-result the initial value of $\dd_{2,0}$ is irrelevant, as it creates only a constant factor. So we may set it to $\dd_{2,0}=1$, or any convenient constant.
Note that this recurrence is very similar to that of relaxed trees~\cite[Equation~(2)]{ElFaWa}, yet with more complicated factors. Observe also that this is exactly recurrence~\cite[Equation~(10)]{FuYuZh} for $d=2$.

\newcommand{\bb}{B}

Motivated by experiments for large $n$, we use the following ansatz
\begin{align*}
    \dd_{n,m} \approx h(n) f\left(\frac{m+1}{n^{1/3}}\right),
\end{align*}
where $h$ and $f$ are some ``regular'' functions. Next, we substitute $s(n)=h(n)/h(n-1)$ and $m=\kappa n^{1/3}-1$ into~\eqref{eq:recdnm}. Then, for $n\to\infty$ we get the expansion
\begin{align*}
    f(\kappa)s(n) &= 2 f(\kappa) + \left(f''(\kappa) - \frac{2(d-1)}{d+1} \kappa f(\kappa)\right) n^{-2/3} + \LandauO\left(n^{-1}\right).
\end{align*}
Hence, we may assume that
$
    s(n) = 2 + c_1 n^{-2/3} + c_2 n^{-1} + \dots
$
and this implies that $f(\kappa)$ satisfies the differential equation
\[
    f''(\kappa) = \left(c_1 + \frac{2(d-1)}{d+1}\kappa\right)f(\kappa)
\]
that is solved by the Airy function $\Ai$ of the first kind. Additionally, the boundary conditions allow to compute $c_1$ and we get that
\[
    f(\kappa) = C \Ai\left( a_1 + \bb^{1/3} \kappa \right)
    \qquad
    \text{ where }
    \qquad
    \bb: = \frac{2(d-1)}{d+1},
\]
$a_1 \approx 2.338$ is the largest root of the Airy function $\Ai$, and $C$ is an arbitrary constant.
From this we get that $c_1 = a_1 \bb^{1/3}$.
These heuristic arguments guide us to the following results.
The proofs are analogous to~\cite{FuYuZh,ElFaWa,ElFaWa2020}; for the details we refer to the accompanying Maple worksheet~\cite{Wa}.

\begin{pro}
	\label{lem:AiryXLower}
	For all $n,m\geq0$ let
	\begin{align*}	
		\tilde{X}_{n,m} &:= \left(1-\frac{2d-1}{3(d+1)}\frac{m^2}{n} - \frac{3d^2+12-11}{6(d+1)}\frac{m}{n}\right)\Ai\left(a_{1}+\frac{\bb^{1/3}(m+1)}{n^{1/3}}\right)~~~~~~~~\text{and}\\
		 \tilde{s}_n &:= 2+\frac{a_1 \bb^{2/3}}{n^{2/3}} - \frac{3d^2-5d+4}{3(d+1)n} - \frac{1}{n^{7/6}}.
	\end{align*}
	Then, for any $\varepsilon>0$, there exists an $\tilde{n}_0$ such that
	\begin{align*}
		\tilde{X}_{n,m}\tilde{s}_{n} \leq \mu_{n,m}^{(d)} \tilde{X}_{n-1,m+1} \!+\! \nu_{n,m}^{(d)} \tilde{X}_{n-1,m-1}
	\end{align*}
	for all $n\geq \tilde{n}_0$ and for all $0 \leq m < n^{2/3-\varepsilon}$, where $\mu^{(d)}_{n,m}$ and $\nu^{(d)}_{n,m}$ are as in Lemma~\ref{mod-rec}.
\end{pro}

\begin{pro}
	\label{lem:AiryXUpper}
	Choose $\eta > \frac{(2d-1)^2}{18(d+1)^2}$ fixed and for all $n,m \geq 0$ let
	\begin{align*}
		\hat{X}_{n,m} &:= \left(1 - \frac{2d-1}{3(d+1)}\frac{m^2}{n} - \frac{3d^2+12-11}{6(d+1)}\frac{m}{n} + \eta\frac{m^4}{n^2}\right)\Ai\left(a_{1}+\frac{\bb^{1/3}(m+1)}{n^{1/3}}\right)~~~~~~~~\text{and}\\
		\hat{s}_n &:= 2 + \frac{a_1 \bb^{2/3}}{n^{2/3}} -  \frac{3d^2-5d+4}{3(d+1)n} + \frac{1}{n^{7/6}}.
	\end{align*}
	Then, for any $\varepsilon>0$, there exists a constant $\hat{n}_0$ such that
	\begin{align*}
		\hat{X}_{n,m}\hat{s}_{n} \geq \mu_{n,m}^{(d)} \hat{X}_{n-1,m+1} \!+\! \nu_{n,m}^{(d)} \hat{X}_{n-1,m-1}
	\end{align*}
	for all $n\geq \hat{n}_0$ and all $0 \leq m < n^{1-\varepsilon}$.
\end{pro}

\begin{proof}[Proof of Theorem~\ref{gen-tc-max-k}]
    Let us start with the lower bound. We first define a sequence $X_{n,m} : =\max\{\tilde{X}_{n,m},0\}$ which satisfies the inequality of Proposition~\ref{lem:AiryXLower} for \emph{all} $m \leq n$.
    Then, we define an explicit sequence $\tilde{h}_n := \tilde{s}_n \tilde{h}_{n-1}$ for $n >0$ and $\tilde{h}_0 = \tilde{s}_0$.
    From this, we show by induction that $\dd_{n,m} \geq C_0 \tilde{h}_{n} X_{n,m}$ for some constant $C_0>0$ and all $n \geq \tilde{n}_0$ and all $0 \leq m \leq n$. Hence,
   \begin{align*}
       \dd_{2n,0}
       &\geq C_0 \tilde{h}_{2n} X_{2n,0}\\
       &\geq C_0 \prod_{i=1}^{2n} \left(2+\frac{a_1 \bb^{2/3}}{i^{2/3}} - \frac{3d^2-5d+4}{3(d+1)i} - \frac{1}{i^{7/6}} \right) \Ai\left(a_{1}+\frac{\bb^{1/3}}{n^{1/3}}\right) \\
       &\geq C_1 (n!)^{d-1} 4^n e^{3 a_1 \bb^{2/3} n^{1/3}} n^{\frac{d^2+d-2}{2(d+1)}}.
   \end{align*}
   Finally, combining this with~\eqref{eq:thetachain} we get the lower bound.

   The upper bound is similar, yet more technical.
   The starting point is Proposition~\ref{lem:AiryXUpper} and a function $X_{n,m}$ that is valid for \emph{all} $0 \leq m \leq n$.
   For this purpose we define a sequence $\ddh_{n,m}$ such that $\ddh_{n,m} := \dd_{n,m}$ for $0 \leq m \leq n^{1-\varepsilon}$ and $\ddh_{n,m} := 0$ otherwise; compare with~\cite{ElFaWa2020,ElFaWa}.
   Then, using tools from lattice path theory and computer algebra, we show that $\dd_{2n,0} = \LandauO(\ddh_{2n,0})$ and that
   \begin{align*}
       \ddh_{2n,0}
       &\leq \hat{C}_1 (n!)^{d-1} 4^n e^{3 a_1 \bb^{2/3} n^{1/3}} n^{\frac{d^2+d-2}{2(d+1)}}.
   \end{align*}
  For more details, see the journal version of this work. 
\end{proof}

\section{Conjectures and Remarks}\label{con+rem}

The main purpose of this paper was to extend recent results on bicombining tree-child networks to $d$-combining tree-child networks. We did this for both one-component tree-child networks as well as general tree-child networks. For one-component $d$-combining tree-child networks, we proved an exact counting formula for their number with $n$ leaves and $k$ reticulation nodes. As a consequence of this formula, we obtained limit laws for the number of reticulation nodes of a random network and (asymptotic) counting results for their total number. For general $d$-combining tree-child networks, our knowledge of their counts is less complete. We derived (asymptotic) results for a fixed number of reticulation nodes and the maximal number of reticulation nodes. The latter implied also an (asymptotic) counting results for their total number.

How about limit laws for the number of reticulation nodes of general $d$-combining tree-child networks? We think that the upper bound (\ref{tc-number-ub}) is sharp for $d=2$ and far away from being sharp for $d\geq 3$. More precisely, we believe that the following conjecture holds.

\begin{conj}\label{ll-gen-tc}
Let $T_n^{(d)}$ be the number of reticulation nodes of a $d$-combining tree-child network picked uniformly at random from the set of all $d$-combining tree-child networks with $n$ leaves. Then, we have the following limit behavior of $T_n^{(d)}$.

\begin{itemize}
\item[(i)] For $d=2$ (bicombining case), we have the weak convergence result:
\[
n-1-T_n^{(2)}\stackrel{w}{\longrightarrow}\operatorname{Poisson}(1/2),
\]
where $\operatorname{Poisson}(\alpha)$ denotes the Poisson distribution.
\item[(ii)] For $d\geq 3$, the limit distribution of $n-1-T_n^{(d)}$ is degenerate.
\end{itemize}
\end{conj}

Moreover, the proof of this conjecture should also give the following result.

\begin{cor}
\begin{itemize}
\item[(i)] For $d=2$ (bicombining case), we have $\mathrm{TC}_n^{(2)}\sim\sqrt{e}\cdot\mathrm{TC}_{n,n-1}^{(2)}$.
\item[(ii)] For $d\geq 3$, we have $\mathrm{TC}_n^{(d)}\sim\mathrm{TC}_{n,n-1}^{(d)}$.
\end{itemize}
\end{cor}

\begin{rem}
Note that even with the above result, it is still not possible to give the first-order asymptotics of $\mathrm{TC}_n^{(d)}$ since the approach of \cite{ElFaWa} is only capable of giving a Theta-result.
\end{rem}

In fact, we recently found a method which should allow us to prove these results; details are currently checked. The proofs (if correct) will be presented in the journal version of this paper.

The above limit distribution result would clarify the behavior of the number of general $d$-combining tree-child networks for a large number, i.e., a number close to $n$, of reticulation nodes. So, how about the remaining range? (Recall that the number of networks for a small number of reticulation nodes is covered by Theorem~\ref{gen-tc-fixed-k}.)

In this regard, there is a recent interesting conjecture for the bicombining case; see \cite{PoBa}. To give details, denote by $\mathcal{C}_{n,k}^{(2)}$ a class of words which is similar defined as in Definition~\ref{words-tc} but with the difference that only $k$ letters appear $3$ times while the remaining $n-k$ letters appear $2$ times. Let $c_{n,k}^{(2)}$ be their number. Then, it was conjectured in \cite{PoBa}, together with some striking consequences, that
\[
\mathrm{TC}_{n,k}^{(2)}=\frac{n!}{(n-k)!}c_{n,k}^{(2)}.
\]
Can this be extended to $d$-combining networks? (The obvious generalization by replacing $2$ by $d$ does not work.) We do not know the answer to this question yet. However, we recently managed to define a modification of $\mathcal{C}_{n,k}^{(d)}$ which can be used to encode $d$-combining tree-child networks. This encoding seems to be useful for the proof of Conjecture~\ref{ll-gen-tc} and might also shed further light on \cite{PoBa}. Details will be again discussed in the journal version.

Finally, how about extension of our results to multicombining tree-child networks, i.e., tree-child networks where different reticulation nodes may have different number of parents? We think that many of the results of this extended abstract can be generalized to this case, however, notation becomes more cumbersome. We might also include a discussion on this in the journal version of the current paper.

\begin{landscape}
\section*{Appendix}\label{app}

\vspace*{0.6cm}
\begin{table}[!htb]
\centering
\begin{tabular}{c|cccccccc} \hline
\multicolumn{1}{c|}{$n\backslash{}k$} & \multicolumn{1}{c}{0} & \multicolumn{1}{c}{1} & \multicolumn{1}{c}{2} & \multicolumn{1}{c}{3} & \multicolumn{1}{c}{4} & \multicolumn{1}{c}{5} & \multicolumn{1}{c}{6} & \multicolumn{1}{c}{7} \\ \hline
2 & 1 & 2 \\
3 & 3 & 21 & 42 \\
4 & 15 & 228 & 1272 & 2544 \\
5 & 105 & 2805 & 30300 & 154500 & 309000 \\
6 & 945 & 39330 & 696600 & 6494400 & 31534200 & 63068400 \\
7 & 10395 & 623385 & 16418430 & 241204950 & 2068516800 & 9737380800 & 19474761600 \\
8 & 135135 & 11055240 & 405755280 & 8609378400 & 113376463200 & 920900131200 & 4242782275200 & 8485564550400 \\
 \hline
\end{tabular}
\caption{\label{tab:TCN_d2}$\mathrm{TC}^{(2)}_{n,k}$ for $2\leq n\leq 8$ and $0 \leq k < n$; see also \cite{CaZh}.}
\end{table}

\vspace*{0.6cm}
\begin{table}[!htb]
\centering
\begin{tabular}{c|ccccccc} \hline
\multicolumn{1}{c|}{$n\backslash{}k$} & \multicolumn{1}{c}{0} & \multicolumn{1}{c}{1} & \multicolumn{1}{c}{2} & \multicolumn{1}{c}{3} & \multicolumn{1}{c}{4} & \multicolumn{1}{c}{5} & \multicolumn{1}{c}{6} \\ \hline
2 & 1 & 2 \\
3 & 3 & 33 & 150 \\
4 & 15 & 492 & 7908 & 55320 \\
5 & 105 & 7725 & 291420 & 6179940 & 57939000 \\
6 & 945 & 132030 & 9603270 & 430105320 & 11292075000 & 132120450000 \\
7 & 10395 & 2471805 & 307525050 & 24586633890 & 1284266876760 & 40079165452200 & 560319972030000 \\ \hline
\end{tabular}
\caption{\label{tab:TCN_d3}$\mathrm{TC}^{(3)}_{n,k}$ for $2\leq n\leq7$ and $0 \leq k < n$.}
\end{table}

\begin{table}[!htb]
\centering
\begin{tabular}{c|cccccc} \hline
\multicolumn{1}{c|}{$n\backslash{}k$} & \multicolumn{1}{c}{0} & \multicolumn{1}{c}{1} & \multicolumn{1}{c}{2} & \multicolumn{1}{c}{3} & \multicolumn{1}{c}{4} & \multicolumn{1}{c}{5} \\ \hline
2 & 1 & 2 \\
3 & 3 & 48 & 546 \\
4 & 15 & 942 & 45132 & 1243704 \\
5 & 105 & 18375 & 2394360 & 227116260 & 11351644920 \\
6 & 945 & 375705 & 107314200 & 23919407460 & 3724353682560 & 291451508298720 \\ \hline
\end{tabular}
\caption{\label{tab:TCN_d4}$\mathrm{TC}^{(4)}_{n,k}$ for $2\leq n\leq 6$ and $0 \leq k < n$.}
\end{table}

\begin{table}[!htb]
\centering
\begin{tabular}{c|ccccc} \hline
\multicolumn{1}{c|}{$n\backslash{}k$} & \multicolumn{1}{c}{0} & \multicolumn{1}{c}{1} & \multicolumn{1}{c}{2} & \multicolumn{1}{c}{3} & \multicolumn{1}{c}{4} \\ \hline
2 & 1 & 2 \\
3 & 3 & 66 & 2016 \\
4 & 15 & 1650 & 242496 & 28710864 \\
5 & 105 & 39135 & 17566470 & 7876446840 & 2307919133520 \\ \hline
\end{tabular}
\caption{\label{tab:TCN_d5}$\mathrm{TC}^{(5)}_{n,k}$ for $2\leq n\leq 5$ and $0 \leq k < n$.}
\end{table}

\begin{table}[!htb]
\centering
\begin{tabular}{c|ccccc} \hline
\multicolumn{1}{c|}{$n\backslash{}k$} & \multicolumn{1}{c}{0} & \multicolumn{1}{c}{1} & \multicolumn{1}{c}{2} & \multicolumn{1}{c}{3} & \multicolumn{1}{c}{4} \\ \hline
2 & 1 & 2 \\
3 & 3 & 87 & 7524 \\
4 & 15 & 2700 & 1246740 & 676431360 \\
5 & 105 & 76515 & 118491090 & 262058953860 & 483098464854720 \\ \hline
\end{tabular}
\caption{\label{tab:TCN_d6}$\mathrm{TC}^{(6)}_{n,k}$ fo $r2\leq n\leq 5$ and  $0 \leq k < n$.}
\end{table}
\end{landscape}

\end{document}